\documentclass[12pt]{amsart}
\usepackage{fullpage,verbatim,amssymb}
\usepackage{hyperref, bbm}
\usepackage{soul}
\usepackage[active]{srcltx}
\usepackage[usenames,dvipsnames]{color}
\hypersetup{colorlinks=true,linkcolor=blue,citecolor=magenta}

\makeatletter
\let\@@pmod\mod
\DeclareRobustCommand{\mod}{\@ifstar\@pmods\@@pmod}
\def\@pmods#1{\mkern4mu({\operator@font mod}\mkern 6mu#1)}
\makeatother

\definecolor{blue}{rgb}{0,0,1}
\definecolor{red}{rgb}{1,0,0}
\definecolor{green}{rgb}{0,.6,.2}
\definecolor{purple}{rgb}{1,0,1}

\long\def\red#1\endred{\textcolor{red}{#1}}
\long\def\blue#1\endblue{\textcolor{blue}{#1}}
\long\def\purple#1\endpurple{\textcolor{purple}{ #1}}
\long\def\green#1\endgreen{\textcolor{green}{#1}}

\newcommand{\g}{\gamma}
\newcommand{\og}{\overset{o}{g}}
\newcommand{\oh}{\overset{o}{h}}
\newcommand{\ok}{\overset{o}{k}}

\newcommand{\C}{\mathbb{C}}

\newcommand{\HH}{\mathfrak{H}}

\DeclareMathOperator{\SL}{SL}
\DeclareMathOperator{\PSL}{PSL}

\newcommand{\sm}{\left(\begin{smallmatrix}}
\newcommand{\esm}{\end{smallmatrix}\right)}
\newcommand{\bpm}{\begin{pmatrix}}
\newcommand{\ebpm}{\end{pmatrix}}

\newtheorem{theorem}{Theorem}
\newtheorem{lemma}[theorem]{Lemma}
\newtheorem{proposition}[theorem]{Proposition}

\theoremstyle{remark}

\numberwithin{theorem}{section}
\numberwithin{equation}{section}

\title{Period-like polynomials for $L$-series associated with half-integral weight cusp forms}
\author{James Branch} 
\address{University of Nottingham}
\email{james.branch@nottingham.ac.uk}
\author{Nikolaos Diamantis} 
\address{University of Nottingham}
\email{nikolaos.diamantis@nottingham.ac.uk}
\author{Wissam Raji}
\address{American University of Beirut and Number Theory Research Unit at Center for Advanced Mathematical Sciences}
\email{wr07@aub.edu.lb}
\author{Larry Rolen}
\address{Vanderbilt University}
\email{larry.rolen@vanderbilt.edu}

\begin{document}
\begin{abstract}
Given the $L$-series of a half-integral weight cusp form, we construct polynomials 
behaving similarly to the classical period polynomial of an integral weight cusp form. We also define a lift of half-integral weight cusp forms to integral weight modular forms that is compatible with the $L$-series of the respective forms. 
\end{abstract}

\maketitle
\section{Introduction}

The Dirichlet series associated by Shimura to half-integral weight modular forms in the last section of his original paper \cite{Sh} has not received as much attention as its integral weight counterpart. Partly because of its failure to possess an Euler product, it has not been extensively studied from arithmetic and algebraic perspectives that have a long history in the case of integral weight modular forms.

In view of this, the two main purposes of this note are: (i) to attach Eichler integrals and period-like polynomials to the $L$-series of a half-integral weight cusp form. This will lead to cohomology classes with coefficients in a \emph{finite}  dimensional vector space in a way that parallels the Eichler cohomology of the integral weight case.
(ii) to define a lift of half-integral weight cusp forms to integral weight modular forms that is compatible with the $L$-series of the respective forms. 

This section presents special cases of the two main results of the note. The first one is obtained from Theorem \ref{Eich} in the special case $N=4$, $k$ such that $4|(k-\frac{5}{2})$ and $a=k-2$:
\begin{theorem} \label{PPintro}Let $k \in \frac12+\mathbb Z$ such that  $k-\frac{5}{2} \in 4 \mathbb N$ and let $f$ be a cusp form of weight $k$ for $\Gamma_0(4)$ such that $f(-1/(4z))=(-2iz)^{k}f(z).$ For each $z$ in the upper half-plane $\mathfrak H$ define the ``Eichler integral"
$$F(z)=\Gamma(k-1)\int_z^{i \infty} f(w) \left ( \sum_{n=0}^{k-\frac{5}{2}} \left (\frac{(4i)^n}{n! \Gamma(k-1-n)}+\frac{4^{k-n-\frac{11}{4}}i^{\frac{5}{2}-n} w^{-\frac{1}{2}}}{(k-\frac{5}{2}-n)!\Gamma(n+\frac{3}{2})}\right ) z^nw^{k-2-n} \right ) dw.$$ Then, 
\begin{enumerate} 
\item For each $z \in \mathfrak H$, $P(z):=F(z)-F(-1/(4z))(-2iz)^{k-\frac{5}{2}}$ is a polynomial of degree at most $k-\frac{5}{2}$ in $z.$
\item We have $$P(z) =i^{k-1}\Gamma(k-1)\sum_{n=0}^{k-\frac{5}{2}}\left (\frac{4^n \Lambda(f, k-1-n)}{n! \Gamma(k-1-n)} + (-1)^{n+1} \frac{4^{k-n-\frac{11}{4}} \Lambda(f, k-\frac{3}{2}-n)}{(k-\frac{5}{2}-n)!\Gamma(n+\frac{3}{2})} \right)z^n,$$
where $\Lambda(f, s)$ is the $L$-series of $f$ (to be defined precisely in the next section). 
\end{enumerate}
\end{theorem}
The polynomial $P$ shares some of the defining features of the \emph{period polynomial} of integral weight forms: it encodes certain values of $\Lambda(f, s)$ inside the interval $[1, k-1]$ and satisfies one of the period relations. 
Indeed, in Prop. \ref{Brug} we will show that $P$ matches exactly the $(k-\frac{3}{2})$-th partial sum in the Taylor expansion of the (``symmetrised" version of the) Eichler cocycle, as extended to all real weights in \cite{BCD}. Since then the 
period polynomial equals the value at the Fricke involution of an Eichler cocycle (based at $i \infty$), we think of $P$ as an analogue of the period polynomial for half-integral weight cusp forms. The period relation of $P$ (see Th. 3.1 (i)) is not immediately visible from its expression in part (2) of Th. \ref{PPintro}, but it is deduced from its relation with the analogue of the Eichler integral $F(z)$, as is the case with the classical period polynomial.

A first, to our knowledge, attempt to develop a cohomology for $L$-series of half-integral cusp forms was made in \cite{KR}. Its main construction encodes $L$-values inside $[1, k-1]$ and satisfies one of the period relations too. However, the analogue of $P$ in \cite{KR} belongs to an infinite dimensional space whereas $P$ is a polynomial of degree $\le k-\frac{5}{2}.$ Another difference is that, as will be seen in the general form of the theorem in the sequel, our polynomial $P$ can be made to encode values of $\Lambda(f, s)$ at a larger class of finite ``arithmetic sequences" inside $[1, k-1]$.

The second main result presented here is a special case of Theorem \ref{lift2}:
\begin{theorem} \label{Cohintro}Let $k \in \frac12+\mathbb Z$ such that $k-\frac{5}{2} \in 4 \mathbb N$. For each cusp form $f$ of weight $k$ for $\Gamma_0(4)$ such that $f(-1/(4z))=(-2iz)^{k}f(z)$ there exists an \emph{explicit} modular form of (integral) weight $k-\frac12$ and level $4$, given in Th. \ref{expl}, such that, for each odd $n=0, \dots, k-\frac{5}{2}$, we have
\begin{multline}\label{e:lift}
\left (\frac{2^{2n} }{n! \Gamma(k-1-n)} + \frac{2^{2k-5-2n} }{(k-\frac{5}{2}-n)!\Gamma(n+\frac{3}{2})} \right ) \Lambda \left (f, k-\frac{5}{4}-n \right )\\=
\frac{i^{\frac{7}{4}}2^{k-\frac{3}{2}}}{\Gamma(k-1)}\binom{k-\frac{5}{2}}{n}\Lambda \left (g, k-n-\frac{3}{2} \right ).
\end{multline}
\end{theorem}
The main characteristic of the ``lift" induced by Theorem \ref{Cohintro} is that it is compatible, in the sense of eq. \eqref{e:lift}, with the $L$-series of the half-integral weight form and that of the corresponding integral weight forms. On the other hand, there does not seem to be any compatibility with the Hecke action, and the ``lifted'' forms are not explicitly given in terms of Fourier expansions, as was the case of the Shimura lift. One also notices that the weight of our ``lift" is half the weight of the Shimura lift. Because of those differences in their behaviour and the different problems they were each designed to resolve, it does not seem likely that our lift is related to the Shimura lift.

The lack of compatibility with the Hecke action was expected: The identity of Theorem \ref{Cohintro} expresses the ``critical" values of $L$-series of a half-integral weight cusp form directly in terms of $L$-values of \emph{integral} weight forms. If our lift were compatible with the Hecke action, then we could immediately deduce algebraicity results about the $L$-values of some half-integral weight Hecke eigenforms from the corresponding results in the integral weight case. However, algebraic properties so similar to those of the integral-weight $L$-values are not expected for $L$-series of half-integral weight forms. Therefore, additional input is required to derive algebraic information from the translation of half-integral weight $L$-values to integral weight $L$-values provided by Theorem \ref{Cohintro}. 

It was also mentioned above that our lift is not given explicitly via Fourier expansions. Nevertheless, it does have an explicit expression through the explicit inverse of the Eichler-Shimura map given in \cite{PP}. Theorem \ref{expl} allows us to obtain the integral weight lift of a given form $f$ of half-integral weight $k$ from $k-3/2$ values of the $L$-series associated with $f$. To formulate and prove this result we apply, in Section \ref{refo}, a method of explicitly determining the inverse of the Eichler-Shimura map on the full space of modular forms. This method may be of independent interest and it is based on a Haberland-type formula, proved in \cite{PP}, which applies to modular forms that are not necessarily cuspidal.

\section*{Data statement }
The authors confirm that this manuscript has no associated data.

\section*{Ethics declarations  }
The authors declare that there are no conflict of interests.

\section*{Acknowledgements}
The authors are grateful to A. Popa for his valuable feedback and a very careful reading of this work. They further thank V. Pasol and the anonymous referee for helpful comments. Part of the work was done while ND was visiting Max Planck Institute for Mathematics in Bonn and the University of Patras, both of whose hospitality he acknowledges. Research on this work is partially supported by the second author's EPSRC Grant EP/S032460/1. The third named author would like to thank the Center for Advanced Mathematical Sciences (CAMS) at the American University of Beirut for the support of the number theory unit. This work was supported by a grant from the Simons Foundation (853830, LR). The fourth author is grateful for support from a Dean’s Faculty Fellowship from Vanderbilt University, and to the Max Planck Institute for Mathematics for its hospitality and financial support.

\section{Terminology and notation}
We first fix the terminology and the notation we will be using. They will mostly be consistent with those of \cite{Sh} and \cite{KR}.

Let $k \in \frac12+\mathbb Z$ and $N \in 4 \mathbb N.$ We let $\left ( \frac{c}{d} \right )$ be the Kronecker symbol. For an odd integer $d$, we set 
\begin{equation}
\epsilon_d:=\begin{cases} 1 & \text{ if } d\equiv 1\bmod{4}, \\
i & \text{ if } d\equiv 3\bmod{4}, 
\end{cases} 
\end{equation}
so that $\epsilon_d^2 = \left(\frac{-1}{d}\right)$. 
We set the implied logarithm to equal its principal branch so that $-\pi <$arg$(z) \le \pi$. 
We define the action $|_k$ of $\Gamma_0(N)$ on smooth functions $f$ on $\HH$
as follows: 
\begin{equation}\label{e:slashk_halfint}
(f|_k\gamma)(z):= 
\left ( \frac{c}{d} \right ) \epsilon_d^{2k} (cz+d)^{-k} f(\gamma z) \qquad \text{ for all } \gamma=\begin{pmatrix} * & * \\  c & d \end{pmatrix}  \in \Gamma_0(N).
\end{equation}
Further, let
 $W_N = \begin{pmatrix} 0&-1/\sqrt{N}\\ \sqrt{N} & 0\end{pmatrix}$ and $\Gamma_0(N)^* = \langle W_N,\Gamma_0(N)\rangle.$  We set
\begin{equation}\label{e:slashk_halfintW}
(f|_kW_N)(z):= (-i\sqrt{N} z)^{-k} f(-1/(Nz)).
\end{equation}
We extend the action to $\mathbb C[\Gamma_0(N)^* ]$ by linearity. 

For $n \in \mathbb Z$ we let, as usual, 
\begin{equation}\label{e:slashk_int}
(f|_n\gamma)(z):= (cz+d)^{-n} f(\gamma z) \qquad \text{ for all } \gamma=\begin{pmatrix} * & * \\  c & d \end{pmatrix}  \in \SL_2(\mathbb R).
\end{equation}

Let $k \in \frac{1}{2}\mathbb Z$. If $\Gamma$ is either a subgroup of finite index in $\SL_2(\mathbb Z)$ or $\Gamma^*_0(N)$ for some $N$, and $\chi$ a character on $\Gamma$, We set 
\begin{equation}\label{e:slashk_chi}
f|_{k, \chi}\gamma:= \overline{\chi(\gamma)}f|_k \gamma \qquad \text{ for all } \gamma \in \Gamma.
\end{equation}
We will denote the space of modular (resp. cusp) forms of weight $k$ and character $\chi$ for $\Gamma$ by $M_k(\Gamma, \chi)$ (resp. $S_k(\Gamma, \chi)$). If $\chi$ is the trivial character, we write $M_k(\Gamma, \chi)$ (resp. $S_k(\Gamma)$).

For $f, g \in S_k(\Gamma, \chi)$, we define the Petersson scalar product as
$$(f, g)=\int_{\Gamma \backslash \mathfrak H} f(z)\overline{g(z)}y^k \frac{dxdy}{y^2} \quad \mbox{where} \, \, z=x+iy.$$

Let $\lambda$ be the width of the cusp $\infty$. For integral or half-integral weights $k$, we attach to 
$$f(z) = \sum_{n\geq 0}a_f(n)e^{\frac{2 \pi i nz}{\lambda}} \in M_k(\Gamma, \chi)$$ the \emph{$L$-series} 
$$L_f(s)=\sum_{n\geq 1} \frac{a_f(n)}{n^s}.$$
This is absolutely convergent for $\Re(s) \gg 1$ and can be meromorphically continued to the entire complex plane. Its ``completed'' version is
\begin{equation}\label{Lambda}
    	\Lambda(f, s) = \frac{\Gamma(s) \lambda^s}{(2\pi)^s}\sum_{n\geq 1} \frac{a_f(n)}{n^s} = \int_0^\infty (f(it)-a_f(0))t^s\frac{dt}{t}. 
\end{equation}
If $f$ is a cusp form $\Lambda$ is entire function and, if, further, $\chi$ trivial, it satisfies the functional equation 
\begin{equation}\label{e:FE}
\Lambda(f, s)=N^{\frac{k}{2}-s}\Lambda(f|_kW_N, k-s), \,\, \, \, \text{if $k \in \frac12+\mathbb Z$ and} \, \,
\Lambda(f, s)=i^k N^{\frac{k}{2}-s}\Lambda(f|_kW_N, k-s), \,\, \, \text{if $k \in \mathbb Z$.}
\end{equation}
We will also need the completed $L$-series of an integral weight form $f$ at the cusp $\gamma \infty$, for some $\g \in \text{SL}_2(\mathbb Z)$.
For $\Re(s) \gg 1$, 
\begin{equation}\label{Lambdacusp}
    	\Lambda(f|_k\g, s) := \int_0^\infty ((f|_k\g)(it)-a_{f|_k\g}(0))t^s\frac{dt}{t}
\end{equation}
where $a_{f|_n\g}(0)$ denotes the constant term of Fourier expansion of $f|_n\g$ at infinity.
Its meromorphic continuation to the entire plane is given by the equality
\begin{multline}\label{Lambdacuspdecom}
    	\Lambda(f|_k\g, s) = \int_{t_0}^\infty ((f|_k\g)(it)-a_{f|_k\g}(0))t^s\frac{dt}{t}+
i^k\int_{1/t_0}^\infty ((f|_k\g S)(it)-a_{f|_k\g S}(0))t^{k-s}\frac{dt}{t} \\-
a_{f|_k\g}(0)\frac{t_0^s}{s}-i^ka_{f|_k\g S}(0)\frac{t_0^{s-k}}{k-s}
\end{multline}
for any $t_0>0,$
where $
S:= \sm 0&-1\\ 1 & 0\esm.$

Finally, we set
$T: = \sm 1&1\\ 0 & 1\esm$ and $U:=TS= \sm 1&-1 \\ 1 & 0\esm.$
For $z, w$ not necessarily non-negative integers, set
$$\binom{z}{w}:=\frac{\Gamma(z+1)}{\Gamma(w+1)\Gamma(z-w+1)}.$$

\section{An analogue of the period polynomial}
Let $f$ be a cusp form of weight $k \in \frac12+\mathbb Z$ for $\Gamma_0(N)$ such that $f|_kW_N=f$. Fix $a\in[0,2k-9/2]$ and set
	\[
		P_a(z) := \int_0^{i\infty} f(w)\Phi_a(z,w)dw,
	\]
where
\begin{equation*}
	\Phi_a(z,w) :=\sum_{n=0}^{k-5/2} \left[\binom{k-2}{n}(iNz)^nw^{a-n}+\frac{i^k}{\sqrt[4]{N}}\binom{k-2}{n+1/2}(iz)^n(-Nw)^{2k-9/2-a-n}\right].
\end{equation*}
\begin{theorem}\label{Eich}	Let $k \in \frac{1}{2}+\mathbb Z$ with $k>5/2$. Suppose that $f\in S_k(\Gamma^*_0(N))$ and $a\in [0,2k-9/2]$. With the above notation, set
	\[
		F_a(z) := \int_z^{i\infty} f(w)\Phi_a(z,w)dw.
	\]
 \begin{enumerate}
     \item[(i)]
For $z\in\mathfrak{H}$ we have
\begin{equation*}
	-i^{k-5/2} F_a|_{5/2-k}W_N +F_a = P_a.
\end{equation*}
Therefore, if $\chi$ is a character on $\Gamma^*_0(N)$ such that $\chi(W_N)=i^{\frac{5}{2}-k}$, then
\begin{equation}\label{peri}
P_a|_{5/2-k, \chi}(W_N+1)=0.
\end{equation}
In particular, if $4|(k-5/2)$, then 
\begin{equation*}
P_a|_{5/2-k}(W_N+1)=0.
\end{equation*}

\item[(ii)] For each $z \in \mathbb C,$
\begin{eqnarray} \label{Lval}
	P_a(z) &=& i^{a+1}\sum_{n=0}^{k-5/2}\left[\binom{k-2}{n} N^n \Lambda(f, a+1-n)\right. \nonumber
 \\ \qquad &+& \left. \binom{k-2}{n+1/2} N^{2k-a-n-\frac{19}{4}} i^{2n-k+\frac{1}{2}} \Lambda(f, 2k-7/2-a-n) \right]z^n.
\end{eqnarray}
 \end{enumerate}
\end{theorem}
\begin{proof}
We first show that
\begin{equation}\label{transf}
	-(i\sqrt{N}z)^{k-5/2}(-i\sqrt{N}w)^{k-2}\Phi_a(W_Nz,W_Nw) = \Phi_a(z,w).
\end{equation}
Indeed, the left-hand side is
\begin{eqnarray*}
		-(i\sqrt{N} z)^{k-5/2}(-i\sqrt{N}w)^{k-2}\sum_{n=0}^{k-5/2}\left[\binom{k-2}{n}(iz)^{-n}(-Nw)^{n-a}\right. \\ \left.+ \frac{i^k}{\sqrt[4]{N}}\binom{k-2}{n+1/2} (iNz)^{-n} w^{n-2k+9/2+a}\right].
\end{eqnarray*}
Observe that
	\[
		\binom{k-2}{k-5/2-n}=\binom{k-2}{n+1/2},
	\]
so that the change of variables $n\mapsto k-5/2-n$ yields
\begin{eqnarray}\label{pretransf}
-(i\sqrt{N}z)^{k-5/2}(-i\sqrt{N}w)^{k-2}\sum_{n=0}^{k-5/2} \left[\binom{k-2}{n+1/2}(iz)^{n-k+5/2}(-Nw)^{k-5/2-n-a}\right. \\ \left. + \frac{i^k}{\sqrt[4]{N}}\binom{k-2}{n}(iNz)^{n-k+5/2}w^{-k+2-n+a}\right].
\end{eqnarray}
Since $w \in \mathfrak H$, we have, for every $t\in\mathbb{R}$,
$(-w)^t = e^{-\pi i t} w^t.$
A routine computation shows
\[
	-(i\sqrt{N}z)^{k-5/2}(-i\sqrt{N}w)^{k-2} (iz)^{n-k+5/2}(-Nw)^{k-5/2-n-a} = \frac{i^k}{\sqrt[4]{N}}(iz)^n (-Nw)^{2k-9/2-a-n}
\]
and similarly:
\[
	-(i\sqrt{N}z)^{k-5/2}(-i\sqrt{N}w)^{k-2} \frac{i^k}{\sqrt[4]{N}} (iN z)^{n-k+5/2}w^{2-k-n+a}=(iNz)^n w^{a-n}.
\]
Therefore, \eqref{pretransf} equals $\Phi_a(z, w)$, establishing \eqref{transf}. 

With \eqref{transf} we now have:
\begin{eqnarray}\label{action}
	i^{k-5/2} (F_a|_{5/2-k} W_N)(z) &=& (i\sqrt{N} z)^{k-5/2}\int_{W_Nz}^{i\infty} f(w)\Phi_a(W_N z,w)dw \nonumber\\
	&=& (i\sqrt{N}z)^{k-5/2} \int_z^0 f(W_Nw)\Phi_a(W_Nz,W_Nw) d(W_Nw) \nonumber\\
	&=& -(i\sqrt{N} z)^{k-5/2} \int_z^0 f(w)(-i\sqrt{N}w)^k \Phi_a(W_Nz,W_Nw)\frac{dw}{(-i\sqrt{N}w)^2} \nonumber\\
	&=& \int_z^0 f(w)\Phi_a(z,w)dw=F_a(z)-P_a(z).
\end{eqnarray}

To prove \eqref{peri}, we have from \eqref{action},
$$P_a|_{\frac{5}{2}-k, \chi}(1+W_N)=F_a|_{\frac{5}{2}-k, \chi}(1-W_N)|_{\frac{5}{2}-k, \chi}(1+W_N)=F_a|_{\frac{5}{2}-k, \chi}(1-W_N^2)=0.
$$

To show \eqref{Lval}, we expand the defining expression for $P_a$ and use the integral formula for $\Lambda(f, s)$ in \eqref{Lambda}. 
\end{proof}

This theorem and \eqref{e:FE} show that $P_a$ can be thought of as a ``period polynomial" encoding the $L$-values of $f$ at $a+1,a,a-1,\cdots, a-k+7/2$. Among the various choices of $a$, the most ``canonical" is $a=k-2$ because then, our ``period polynomial" $P_a$ becomes entirely consistent with the Eichler cohomology attached to general weight cusp forms, as in \cite{BCD}. Specifically, the Eichler cocycle on which that cohomology is based is induced by the assignment (Section 2.2 of \cite{BCD}) to $f \in S_k(\Gamma)$ of the map
$$ \Gamma \ni \gamma \longrightarrow \psi_{f, \gamma}^{\infty}(z):=\int_{\gamma^{-1}i\infty}^{i \infty} f(w)(w-z)^{k-2}dw$$
where $\psi_{f, \gamma}^{\infty}$ is defined on the \emph{lower} half-plane $\bar {\mathfrak H}$. 
For $\gamma=W_N$, the integral giving $\psi_{f, W_N}^{\infty}\left (x \right )$ is well-defined for $x>0$ as well. We then have the following relation between our ``period polynomial" $P_{k-2}$ and $\psi_{f, W_N}^{\infty}\left (x \right )$. 
\begin{proposition}\label{Brug}
For each $x>1$, we have 
$$P_{k-2}(ix)= \psi_{f, W_N}^{\infty}\left (Nx \right )-i^{1-2k}\psi_{f, W_N}^{\infty}\left (1/x \right )(\sqrt{N} x)^{k-\frac{5}{2}}+O(x^{k-\frac32}).
$$
\end{proposition}
\begin{proof}We first note that
\begin{multline}
\psi_{f, W_N}^{\infty}\left (Nx \right )-i^{1-2k}\psi_{f, W_N}^{\infty}\left (1/x \right )(\sqrt{N} x)^{k-\frac{5}{2}} \\=
i^{k-1}\int_0^{\infty} f(it) \left ( (1+iNxt^{-1})^{k-2} -i^{1-2k}(1+ix^{-1}t^{-1})^{k-2} (\sqrt{N} x)^{k-\frac{5}{2}}\right ) t^{k-2}dt.
\end{multline}
By Taylor's formula we have, for each $M \in \mathbb N$, 
\begin{multline*}
    (1+iNxt^{-1})^{k-2}= \\ \sum_{n=0}^{M-1} \binom{k-2}{n}(iNxt^{-1})^n+O_M\left (\int_0^1(1-y)^{M-1}(1+iNyxt^{-1})^{k-2-M} (iNxt^{-1})^M dy\right ).\end{multline*}
If $M>k-2$ the error term is $O_M((xt^{-1})^M)$. Likewise, again if $M>k-2$, the error term of $(1+ix^{-1}t^{-1})^{k-2}$ is $O_M((x^{-1}t^{-1})^M)$.  Therefore for $M=k-\frac32$, we deduce, for $x>1$,
\begin{multline}
\psi_{f, W_N}^{\infty}\left (Nx \right )-i^{1-2k}\psi_{f, W_N}^{\infty}\left (1/x \right )(\sqrt{N} x)^{k-\frac{5}{2}} \\=
i^{k-1}\sum_{n=0}^{k-\frac{5}{2}} \binom{k-2}{n} \left ((iNx)^n-i^{1-2k} (ix^{-1})^n (\sqrt{N} x)^{k-\frac{5}{2}} \right ) \Lambda(f, k-1-n)+O(x^{k-\frac{3}{2}})
\end{multline}
where the implied constant is independent of $x$. The change of variables $n \to k-\frac{5}{2}-n$, followed by \eqref{e:FE} in the second sum, implies the result.
\end{proof}

\subsection{Comparison with the period function of \cite{KR}}
Another function encoding special values of $L_f$ is given in \cite{KR}. The result in \cite{KR} is stated for cusp forms on Hecke groups, but, thanks to the embedding of $S_k(\Gamma^*_0(N))$ into a space of cusp forms for Hecke groups ((8.1) of \cite{KR}) it can be formulated for the modular forms studied in Theorem \ref{Eich}: 
\begin{proposition} {\bf \cite{KR}} Let $k \in \frac12+\mathbb Z$ and $f(z)=\sum a_f(n)e^{2 \pi i n z} \in S_k(\Gamma_0(N))$ such that $f|_kW_N=f$. For each $z \in \HH$ 
set
$$\mathcal E^*_f(z)=\frac{1}{\sqrt{\pi}}\sum_{n \ge 1} \frac{a_f(n)}{n^{k-1}}\left ( e^{-2\pi i n z} \Gamma \left (\frac12, -2\pi i n z \right )-\frac{1}{\sqrt{-2\pi i n z}} \right ).$$
Then, for all $z \in \mathfrak H,$
\begin{equation}\label{KREich}\left (\mathcal E^*_f|_{2-k}(1-W_N)\right ) (z)=\sum_{n=0}^{k-\frac{3}{2}} \left (\frac{L_f(k-n-1)}{\Gamma(n+1)}+ \frac{L_f(k-n-\frac{1}{2})}{\Gamma(n+\frac{1}{2})}\left  (\frac{2\pi z}{i}\right )^{-\frac12}\right ) \left  (\frac{2\pi z}{i}\right )^n.\end{equation}
\end{proposition} 
The values of $L_f(s)$ at 
$k-1, \dots \frac32$ appearing in the right-hand side of \eqref{KREich} are also encoded by $P_{k-2}$. However, the function on the right-hand side of \eqref{KREich} does not belong to a finite-dimensional space closed under the action of the group. Further, the ``Eichler integral" 
$\mathcal E^*_f(z)$ is defined as a series. To complete the comparison of our construction with the ``period function" of \cite{KR}, we show how the main piece of $\mathcal E^*_f(z)$ can nevertheless be expressed as an integral too. 

\begin{proposition} With the notation of Theorem \eqref{Eich}, for each $z \in \mathfrak H$,
$$\mathcal E^*_f(z)=\alpha_k (-iz)^{\frac12} \int_{z}^{i \infty} F_f(w)
(w-z)^{k-\frac52} dw-\frac{1}{\sqrt{-2 \pi^2 i z}}L_f\left (k-\frac12\right )$$
where $$F_f(w):= \int_0^{\infty} f\left (xw\right )x^{k-\frac32}(x+1)^{-\frac12} dx $$
and $\alpha_k=(-2 \pi i)^{k-1}/ (\pi^{\frac12}(k-\frac52)!).$
\end{proposition}
\begin{proof} We first recall (\cite{NIST}, (8.19.1), (8.19.3)) that, for Re$(w)>0,$
$$\Gamma \left (\frac12, w \right )=w^{\frac12}\int_1^{\infty}e^{-wt} t^{-\frac12} dt.$$
Therefore, for $z \in \mathfrak H,$
\begin{multline*} \sum_{n \ge 1} \frac{a_f(n)}{n^{k-1}}\frac{e^{-2\pi i n z }}{\sqrt{\pi}} \Gamma \left (\frac12, -2\pi i n z \right )=
\frac{1}{\sqrt{\pi}} \sum_{n \ge 1} \frac{a_f(n)}{n^{k-1}} e^{-2\pi i n z } \left (-2\pi i n z \right)^{\frac12}\int_1^{\infty} e^{2\pi i n zt }t^{-\frac12} dt\\
=(-2 i z)^{\frac12} \int_1^{\infty}t^{-\frac12} \left ( \sum_{n \ge 1} \frac{a_f(n)}{n^{k-\frac32}} e^{2 \pi i n z(t-1)} \right ) dt.
\end{multline*}
By the theory of usual (integral weight) Eichler integrals, followed by the changes of variables $x=t-1$ and $w_1=w/x$, this equals
\begin{multline*} (-2 i z)^{\frac12} \int_1^{\infty}t^{-\frac12} \frac{(-2 \pi i)^{k-\frac32}}{\Gamma(k-\frac32)}
\int_{z(t-1)}^{i \infty} f(w) \left (w-z(t-1)\right )^{k-\frac52} dwdt\\
\\=
\alpha_k z^{\frac12}
 \int_0^{\infty}x^{k-\frac32} (x+1)^{-\frac12} \int_{z}^{i \infty} f(xw_1)\left (w_1-z \right )^{k-\frac52} dw_1 dx.
\end{multline*}
A change in the order of integration implies the formula.
\end{proof}

\section{An explicit lift to spaces of integral weight}\label{refo}
In this section, we will first use Theorem \ref{Eich} to construct an Eichler cocycle, with coefficients in the space $\mathbb{C}_{k-5/2}[z]$. We maintain the notation and assumptions of the last section. 

Since $4|N$, $\Gamma_0(N)/\{\pm 1\}$ 
is torsion-free and hence free on a set of generators $\{\gamma_j\}_{j=1}^{2g+h-1}$, where $\gamma_1=T$, $g$ is the genus and $h$ the number of inequivalent cusps of $\Gamma_0(N)$. (\cite{Iw}, Prop. 2.4; here by $T, \gamma_j$ etc. we mean the images of those elements of $\Gamma_0(N)$ into $\Gamma_0(N)/\{\pm 1\}$), We also have $W_N\Gamma_0(N)=\Gamma_0(N)W_N$ and thus $\Gamma^*_0(N)=\langle \Gamma_0(N), W_N \rangle$ is generated by $\{\gamma_j\} \cup \{W_N\}$ with only relations $(-1)^2=1$ and $W_N^4=1$. 

From the above, we first deduce that there is always a character $\chi$ on $\Gamma^*_0(N)$ such as prescribed in Theorem \ref{Eich}, that is such that $\chi(W_N)=i^{5/2-k}$. Indeed, since $i^{4(5/2-k)}=1$, the character induced by the assignment $\chi(-1)=(-1)^{5/2-k}$, $\chi(W_N)=i^{5/2-k}$ and $\chi(\gamma_i)=1$ for $i=1, \dots 2g+h-1$ is well-defined. 

Further, let, for $a \in [0, 2k-\frac92]$, $P_a$ be the polynomial defined in \eqref{Lval}. Consider the polynomial 
$$\hat{P}_a(z):=P_a(2z/\sqrt{N}).$$ It is then easy to deduce from Theorem \ref{Eich} that
\begin{equation}\label{P_1}
\hat{P}_a|_{\frac52-k, \chi}(W_4+1)=0
\end{equation}
for any character $\chi$ on $\Gamma^*_0(4)$ such that $\chi(W_4)=i^{\frac{5}{2}-k}$. (Since there is no risk of confusion, we use the same notation for both characters). 
Recall that $\Gamma_0(4)/\{ \pm 1\}$ is freely generated on the generators $$T, \begin{pmatrix} 1 & 0 \\ 4 & 1 \end{pmatrix}$$
and thus $\Gamma^*_0(4)$ is generated by $-1, T, W_4$ in $\Gamma^*_0(4)$ with only relations $(-1)^2=1$ and $W_4^4=1$.
We consider the map
\[
	\hat{\pi}_f : \Gamma^*_0(4) \to \mathbb{C}_{k-5/2}[z]
\]
induced by the $1$-cocycle condition from the values
\[
	\hat{\pi}_f(W_4) = \hat{P}_a \qquad\mbox{ and }\qquad \hat{\pi}_f(-1)=\hat{\pi}_f(T)=0.
\] 
Then $\hat{\pi}_f$ is well-defined since by \eqref{P_1},
\[
	\hat{\pi}_f(W_4^2) = \hat{\pi}_f(W_4)|_{\frac{5}{2}-k, \chi} W_4 +\hat{\pi}_f(W_4) = \hat{P}_a|_{\frac{5}{2}-k}W_4 + \hat{P}_a = 0.
\]

Based on this cocycle, we will construct a parabolic cocycle with coefficients in a different module which is defined in \cite{PP}. To do so we review some notation from \cite{N}.

We first consider the
Hecke group $H(2)$ generated by the images of $T^2$ and $S$ under the natural projection of $\SL_2(\mathbb Z)$ onto PSL$_2(\mathbb Z)$. Since $4|(k-\frac{5}{2})$, it will be legitimate to use the same notation for elements of $\SL_2(\mathbb Z)$ and their images in PSL$_2(\mathbb Z)$. The group $H(2)$ has only the relation $S^2=1$ (see Sect. 5 of \cite{KR} for the summary of some basic properties of the Hecke groups). 

We now recall a construction of \cite{N} that is used to provide an explicit formula for cocycles on $\PSL_2(\mathbb Z)$ induced by cocycles on $H(2)$.
It is easy to see that a set of representatives of $H(2) \backslash \PSL_2(\mathbb Z)$ is
$$\{1, T, U\}, \qquad \mbox{where} \, \,  U=TS.$$
To be able to keep track of coset representatives, we denote by $u:\PSL_2(\mathbb Z)\to \{1,T,U\}$ the map which sends $x$ to its corresponding coset representative in $H(2)\backslash\PSL_2(\mathbb Z)$. Specifically, let $x \in \PSL_2(\mathbb Z)$. If $H(2)x=H(2)$ (resp. $H(2)T, H(2)U$), set $u(x)=1$, (resp. $T, U$). Some values of $u$ that will be repeatedly used below tacitly are: 
\begin{equation}\label{usef}u(T^{-1})=T, u(TST^{-1})=U.
\end{equation}

For each $x, g \in \PSL_2(\mathbb Z)$, set 
\begin{equation}
\label{kappa}\kappa_{x, g}:=u(x)gu(xg)^{-1} \in H(2).
\end{equation}
One notices that, if $x \in H(2)$, $\kappa_{x, g}=gu(g)^{-1}$ and, if $x, g \in H(2)$, $\kappa_{x, g}=g$. Further, if $H(2)x=H(2)x'$, then $\kappa_{x, g}=\kappa_{x', g}$ and thus $\kappa_{x, g}$ is well-defined as a function of $\left ( H(2) \backslash \PSL_2(\mathbb Z) \right ) \times \PSL_2(\mathbb Z).$ Finally, $\kappa$ satisfies the relation 
\begin{equation}
\label{cocykappa}\kappa_{x, g_1g_2}=\kappa_{x, g_1}\kappa_{xg_1, g_2}.
\end{equation}
Next, we consider the space $\mathcal A$ of holomorphic functions on $\HH$ and the space
$$I_k^0:=\{f: H(2) \backslash \PSL_2(\mathbb Z) \to  \mathcal A\}.$$
Since $\PSL_2(\mathbb Z)$ acts on $\mathcal A$, there is an action of $\PSL_2(\mathbb Z)$ on $I_k^0$, given, for $v \in I_k^0$, by
$$(v||g)(x):=v(xg^{-1})|_{\frac{5}{2}-k}g \quad \mbox{for all} \, \, x \in H(2) \backslash \PSL_2(\mathbb Z),\, g \in \PSL_2(\mathbb Z).$$ Further set
$$I_k:=\{f: H(2) \backslash \PSL_2(\mathbb Z) \to z^{-1}\mathbb{C}_{k-1/2}[z]\} \hookrightarrow I_k^0.$$
Although $||$ is not an action on $I_k$, since $z^{-1}\mathbb{C}_{k-1/2}[z]$ is not closed under $|_{5/2-k}$, the space 
$$W:=\{v \in I_k; v||(S+1)=v||(U^2+U+1)=0\}$$
is well-defined. 
As in the classical case, the space $W$ can be decomposed as a direct sum of the $\pm$-eigenspaces of a certain involution. Specifically, let $\epsilon=\sm -1& 0 \\0 &1\esm$ and let $\epsilon$ act on a $v \in I^0_k$ so that
$$(v||\epsilon)(x):=v(\epsilon x \epsilon)|_{\frac{5}{2}-k} \epsilon \quad \mbox{for all} \, \, x \in H(2) \backslash \PSL_2(\mathbb Z).$$
Then $I^0_k$ decomposes into $\pm$-eigenspaces under the action of $\epsilon$
. Further, since $W$ is closed under the action of $\epsilon$, it also decomposes into $\pm$-eigenspaces $W^{\pm}$. 
We denote the $\pm$-component of $\sigma \in W$ by $\sigma^{\pm}$, i.e.,
$$\sigma^{\pm}=\frac{1}{2}\left (\sigma \pm \sigma||\epsilon\right ) \in W^{\pm}.$$ 
For each $x \in H(2) \backslash \PSL_2(\mathbb Z)$, $\sigma^{+}(x)$ (resp. $\sigma^{-}(x)$) is the ``even" (resp. ``odd") part of $\sigma(x)$.

In Propositions 8.1 and 8.4 of \cite{PP} an element of $W$ is associated to each modular form of weight $k-1/2$: Let $f$ be an element of the space $M_{k-1/2}(H(2))$ of \emph{all} modular (not necessarily cuspidal) forms for $H(2)$ of weight $k-1/2$. Let $a_0(f)$ denote the constant term of its Fourier expansion at infinity and set $f^0(t):=f(t)-a_0(f)$. For each $\gamma \in H(2)\backslash \PSL_2(\mathbb Z)$ consider $\tilde f(\gamma) \in \mathcal A$ given by 
$$\tilde f(\gamma)(z):=\int_z^{i \infty}\left (f|_{k-\frac12}\gamma \right )^0(t) (t-z)^{k-\frac52}dt \qquad \text{for $z \in \HH.$}$$
This induces an element $\tilde f$ of $I_k^0$, since $f$ is $H(2)$-invariant,  and we have the following
\begin{proposition}\label{8.18.4} ({\it Prop. 8.1, 8.4 \cite{PP}}) \\
i. The element $\rho_f:=\tilde f||(1-S)$ belongs to $W$ and, for each $\gamma \in H(2) \backslash \PSL_2(\mathbb Z), z_0, z \in \HH$, we have
\begin{multline}\label{rhof} \rho_f(\gamma)(z)= \int_{z_0}^{i \infty} \left (f|_{k-\frac12}\gamma \right )^0(t)(t-z)^{k-\frac52}dt - a_0 \left ( f|_{k-\frac12}\gamma \right ) \int_0^{z_0}(t-z)^{k-\frac52}dt-\\
\left ( \int_{Sz_0}^{i \infty}\left (f|_{k-\frac12}\gamma S \right )^0(t)(t-z)^{k-\frac52}dt 
- a_0 \left ( f|_{k-\frac12}\gamma S \right ) \int_0^{Sz_0}(t-z)^{k-\frac52}dt \right ) \Big |_{\frac52-k}S
\\+\frac{a_0 \left ( f|_{k-\frac12}\gamma \right )}{k-\frac32}z^{k-\frac32}
+\frac{a_0 \left ( f|_{k-\frac12}\gamma S\right )}{k-\frac32}z^{-1}
\end{multline} 
(independently of $z_0.$)\\ 
ii. The map $\rho^-:M_{k-\frac12}(H(2)) \to W^-$ given by $f \to \rho_f^-$ is an isomorphism.
\end{proposition}
{\bf Remark.} Assertion ii. follows from Prop. 8.4(b) because $H(2)$ is $\PSL_2(\mathbb Z)$-conjugate to $\Gamma_0(2)$ and $N=2$ is of the form prescribed in Prop. 4.4 of \cite{PP}.

We will now construct another element of $W^-$ based on a cusp form $f$ of weight $k$ for $\Gamma_0(N)$ such that $f|_kW_N=f$. 
Earlier in this section, we defined the map $\hat{\pi}_f : \Gamma^*_0(4) \to \mathbb{C}_{k-5/2}[z]$ induced by the $1$-cocycle condition from the values $\hat{\pi}_f(W_4) = \hat{P}_a$ and $\hat{\pi}_f(-1)=\hat{\pi}_f(T)=0.$
Since $4|(k-\frac{5}{2})$, the $1$-cocycle condition of $\hat{\pi}_f$ implies that there is a well-defined $1$-cocycle $\pi'_f: H(2)\to \mathbb{C}_{k-5/2}[z]$ which is induced by the $1$-cocycle relation from the values 
\[	\pi'_f(S)(z) = \hat{P}_a(z/2) \qquad\mbox{ and } \pi'_f(T^2)=0.\] 
This cocycle gives a non-trivial class in $H^{1}(H(2), \mathbb{C}_{k-5/2}[z])$ where the action of $H(2)$ on $\mathbb{C}_{k-5/2}[z]$ is $|_{5/2-k}$. 

We now define a $1$-cocycle of $\PSL_2(\mathbb Z)$ with coefficients in $I_k$ induced by the cocycle $\pi'_f$. For each $g \in \PSL_2(\mathbb Z)$ let $\tilde \pi_f(g)$ be the element of $I_k$ such that 
\begin{equation}\label{tpi}
\tilde \pi_f(g)(x):=\pi'_f(\kappa^{-1}_{x, g^{-1}})|_{\frac{5}{2}-k} u(x)\quad \mbox{for all} \, \, x \in H(2) \backslash \PSL_2(\mathbb Z).
\end{equation}
By Shapiro's lemma or, directly with \eqref{kappa}, we can see that $\tilde \pi_f$ is a $1$-cocycle. We will show that it gives a \emph{parabolic} $1$-cocycle with coefficients in $I_k^0.$
\begin{proposition}\label{keyprop} There is a $v \in I_k$ such that the map $\pi_f$ given by
\begin{equation}\label{pi}\pi_f(\gamma):=\tilde \pi_f(\gamma)-v||(\gamma-1), \qquad \text{for $\gamma \in \PSL_2(\mathbb Z)$}
\end{equation}
is a parabolic cocycle for $\PSL_2(\mathbb Z)$ with coefficients in $I_k^0$, and 
\begin{equation}\label{W}
\pi_f(S) \in W.
\end{equation}
\end{proposition}
\begin{proof} We first compute the values of $\tilde \pi_f(T)$ using the definition of $u$, \eqref{usef}, the relations $S^2=U^3=1$ and the cocycle relation for $\pi'_f$: 
\begin{align}\label{Tvalues1} \tilde \pi_f(T)(1)&= \pi'_f(\kappa^{-1}_{1, T^{-1}})=\pi'_f((T^{-1} u(T^{-1})^{-1})^{-1})=\pi'_f(T^{2})=0\\
\nonumber
\tilde \pi_f(T)(T) &= \pi'_f(\kappa^{-1}_{T, T^{-1}}) |_{\frac52-k}T=\pi'_f((T \, T^{-1} u(1)^{-1})^{-1}) |_{\frac52-k}T=
\pi'_f(1) |_{\frac52-k}T=0\\
\nonumber
\tilde \pi_f(T)(U) & = \pi'_f(\kappa^{-1}_{U, T^{-1}})|_{\frac52-k}U=\pi'_f((U T^{-1} u(UT^{-1})^{-1})^{-1})|_{\frac52-k}U=
\pi'_f((U T^{-1} U^{-1})^{-1})|_{\frac52-k}U
\\ \nonumber
&=\pi'_f(S^{-1}T^{-2})|_{\frac52-k}U=\pi'_f(S)|_{\frac52-k}T^{-2}U+\pi'_f(T^{-2})|_{\frac52-k}U=\pi'_f(S)|_{\frac52-k}T^{-2}U.
\end{align}
It is easy to see that there is a polynomial in $\mathbb C_{k-\frac32}[z]$
\begin{equation}\label{Q}Q_f(z)=\sum_{n=0}^{k-\frac32}b_nz^n
\end{equation} 
such that
\begin{equation}\label{poly}
    \pi'_f(S)|_{\frac52-k}T^{-2}U=Q_f|_{\frac52-k}(T-1).
\end{equation}
It is clear that the coefficients of $Q_f$ can determined from those of $P_a$ by solving a linear system. If, further we require the constant term of $Q_f$ to be $0$, then this determination is unique.

Let $v \in I_k$ be given by 
$$v(1)=v(T)=0 \qquad \text{and $v(U)=Q_f$}$$
where $Q_f$ is a polynomial satisfying \eqref{poly} with a zero constant term.
Then, $(v||(T-1))(\gamma)=v(\gamma T^{-1})|_{\frac52-k}T-v(\gamma)$ and hence, with \eqref{usef},
\begin{align*}\label{vvalues} (v||(T-1))(1)&= v( T^{-1})|_{\frac52-k}T-v(1)=v(T)|_{\frac52-k}T-v(1)=0 \\
(v||(T-1))(T) &= v(T T^{-1})|_{\frac52-k}T-v(T)=0\\
(v||(T-1))(U) & = v(U T^{-1})|_{\frac52-k}T-v(U)=v(TS T^{-1})|_{\frac52-k}T-v(U)=v(U)|_{\frac52-k}(T-1)\\ \nonumber
&=\pi'_f(S)|_{\frac52-k}T^2U=\tilde \pi_f(T)(U)
\end{align*}
This and \eqref{Tvalues1} imply that the map $\pi_f$ given by \eqref{pi} satisfies 
\begin{equation}\label{par} \pi_f(T)(x)=\tilde \pi_f(T)(x)-(v||(T-1))(x)=0 \qquad \text{for all $x \in H(2) \backslash \PSL_2(\mathbb Z),$}
\end{equation}
i.e. $\pi_f(T)=0.$ Since $\gamma \to v||(\gamma-1)$ is a coboundary (with coefficients in $I_k^0$), $\pi_f$ remains a $1$-cocycle, but now, with coefficients in $I_k^0$. With \eqref{par} it is a parabolic one. The $1$-cocycle condition implies, in particular, that $\pi_f(S)=\pi_f(U)$ satisfies 
\begin{equation}\label{Manin} \pi_f(S)||(S+1)=\pi_f(S)||(U^2+U+1)=0
\end{equation}

On the other hand $\pi_f(S) \in I_k$. Indeed, $(v||(S-1))(\gamma)=v(\gamma S^{-1})|_{\frac52-k}S-v(\gamma)$ and hence, with \eqref{usef},
\begin{align}\label{Svalues1} (v||(S-1))(1)&= v( S^{-1})|_{\frac52-k}S-v(1)=v(1)|_{\frac52-k}1-v(1)=0 \\
\label{Svalues2}
(v||(S-1))(T) &= v(T S^{-1})|_{\frac52-k}S-v(T)=v(U)|_{\frac52-k}S=Q_f|_{\frac52-k}S \in z^{-1}\mathbb C_{k-\frac12}[z]\\
\label{Svalues3}(v||(S-1))(U) & = v(U S^{-1})|_{\frac52-k}S-v(U)=v(T)|_{\frac52-k}S-v(U)=-Q_f \in z^{-1}\mathbb C_{k-\frac12}[z].
\end{align}
Therefore $\pi_f(S)=\tilde \pi_f(S)-v||(S-1) \in I_k.$ This, together with \eqref{Manin}, imply \eqref{W}.
\end{proof}
\begin{proposition}\label{lift2} Let $k \in \frac{1}{2}+\mathbb Z$ such that $k>\frac{5}{2}$ and $4|(k-\frac{5}{2})$. For each $f \in S_k(\Gamma_0(N))$ such that $f|_kW_N=f$ and for each $a \in [0, 2k-\frac92]$, there is an explicit $g \in M_{k-\frac{1}{2}}(\Gamma^*_0(4))$, given in Th. \ref{expl}, such that, for all odd $n \in \{1, \dots, k-\frac{7}{2}\}$,
\begin{multline}\label{equal2}
i^{a+1}\left (\binom{k-2}{n} N^{\frac{n}{2}} \Lambda(f, a+1-n)+ (-1)^{n+1} \binom{k-2}{n+\frac12} N^{2k-a-\frac{3n}{2}-\frac{19}{4}}\Lambda(f, 2k-\frac72-a-n) \right ) \\=
\binom{k-\frac{5}{2}}{n}i^{-1-n}2^{k-\frac{3}{2}-n}\Lambda \left ( g, k-\frac{3}{2}-n\right ).
\end{multline}
\end{proposition}
\begin{proof}
By \eqref{W} and Prop. \ref{8.18.4}(ii), there exists a $g_1 \in M_{k-\frac12}(H(2))$ such that
\begin{equation}\label{basic}
    \pi_f(S)^{-}=\rho^{-}_{g_1}.
\end{equation} 
Since $H(2)\epsilon T\epsilon=H(2)T$
and $H(2)\epsilon U\epsilon=H(2)T^{-1}S^{-1}=H(2)U$, we see that for each $x \in H(2) \backslash \PSL_2(\mathbb Z)$, $\pi_f(S)^{-}(x)$
(resp. $\rho^{-}_{g_1}(x)$) is the part of the polynomial $\pi_f(S)(x)$ (resp. $\rho_{g_1}(x)$) corresponding to its odd powers. By the definition of $\pi_f$, \eqref{Svalues1} and since $\kappa_{1, S^{-1}}=S^{-1}$, we deduce
\begin{equation}\label{piS}
\pi_f(S)(1)=\tilde \pi(S)(1)-(v||(S-1))(1)=\pi'_f(\kappa^{-1}_{1, S^{-1}})=\pi'_f(S)=\hat P_a(z/2)=P_a(z/\sqrt{N}).
\end{equation}
This and \eqref{Lval} show that the $n$-th coefficient of $\pi_f(S)(1)$ equals the left-hand side of \eqref{equal2}. 

On the other hand, \eqref{rhof}, \eqref{Lambdacuspdecom} and an application of the binomial theorem imply 
\begin{multline}\label{rhofexp} \rho_{g_1}(\gamma)(z)= \sum_{n=0}^{k-\frac52}\binom{k-\frac52}{n}(-z)^n \left \{ \int_{z_0}^{i \infty} \left (g_1|_{k-\frac12}\gamma \right )^0(t)t^{k-\frac52-n}dt - \frac{a_0 \left ( g_1|_{k-\frac12}\gamma \right )z_0^{k-\frac32-n}}{k-\frac32-n}- \right.\\
\left. (-1)^n\left ( \int_{Sz_0}^{i \infty}\left (g_1|_{k-\frac12}\gamma S \right )^0(t)t^{n}dt 
- \frac{a_0 \left ( g_1|_{k-\frac12}\gamma S \right ) (Sz_0)^{n+1}}{n+1} \right ) \right \}
\\+\frac{a_0 \left ( g_1|_{k-\frac12}\gamma \right )}{k-\frac32}z^{k-\frac32}
+\frac{a_0 \left ( g_1|_{k-\frac12}\gamma S\right )}{k-\frac32}z^{-1}
=
\\ \sum_{n=0}^{k-\frac52}\binom{k-\frac52}{n}(-z)^n i^{k-\frac32-n}\Lambda \left(g_1|_{k-\frac12} \gamma, k-\frac32-n \right )
+\frac{a_0 \left ( g_1|_{k-\frac12}\gamma \right )}{k-\frac32}z^{k-\frac32}
+\frac{a_0 \left ( g_1|_{k-\frac12}\gamma S\right )}{k-\frac32}z^{-1}.
\end{multline} 
Therefore, for odd $n$ with $0 \le n \le k-\frac{5}{2}$, the $n$-th coefficient of $\rho_{g_1}(1)$ equals
$$\binom{k-\frac{5}{2}}{n}i^{-1-n}\Lambda\left (g_1, k-\frac{3}{2}-n\right ).$$ Further, by the analogue of the proposition in Section 8 of \cite{KR} for integral weights, 
the function $g$ such that $g(z):=g_1(2z)$ is a weight $k-\frac{1}{2}$ modular form for $\Gamma^*_0(4)$ and $\Lambda(g_1,s)=2^s\Lambda(g, s)$. This gives the expression in the right-hand side of \eqref{equal2}. With \eqref{basic}, we deduce \eqref{equal2}.
\end{proof}

We will identify explicitly the integral weight modular form to which a half-integral weight form is lifted according to Prop \ref{lift2} by using results of \cite{PP}. 

We first decompose the function $g$ of Prop. \ref{lift2} in terms of its ``cuspidal" and ``Eisenstein" part. Specifically, let $E(z)$ denote the weight $k-1/2$ Eisenstein series for $\PSL_2(\mathbb Z)$ normalised so that its constant term of its Fourier expansion is $1$. We consider the Eisenstein series $E(2z)$ and $E(z)+2^{k-1/2}E(4z)$ in $M_{k-\frac12}(\Gamma_0^*(4))$ corresponding to the two cusps of $\Gamma_0^*(4).$ Then
\begin{equation}\label{decom} g=\og+\alpha E(2z)+\beta(E(z)+2^{k-1/2}E(4z))
\end{equation} for some $\og \in S_{k-1/2}(\Gamma^*_0(4))$ and some $\alpha, \beta \in \mathbb C.$ 

To evaluate $\alpha, \beta$, we note that, as above, the function
$\og_1(z):=\og(z/2)$ belongs to $S_{k-1/2}(H(2))$ and, likewise,
$$E_1(z):=E(z) \quad \text{and $E_2(z):=E(z/2)+2^{k-1/2}E(2z)$}$$ are modular of weight $k-1/2$ for the group $H(2).$ Equ. \eqref{basic} implies
\begin{equation}\label{basic2} \pi_f(S)^{-}=\rho^{-}_{\, \, \og_1}+\alpha\rho^{-}_{E_1}
+\beta\rho^-_{E_2},
\end{equation}
With the definition of $\pi'_f$ and \eqref{Svalues1}-\eqref{Svalues3}, we have
\begin{align}\label{indep}
&\pi_f(S)(1)=P_a(z/\sqrt{N}), \\
&\pi_f(S)(T)=\pi'_f(\kappa^{-1}_{T, S^{-1}})|_{\frac{5}{2}-k}T-(v||(S-1))(T)=\pi_f'(1)|_{\frac{5}{2}-k}T-Q_f|_{\frac52-k}S= -Q_f|_{\frac52-k}S \nonumber\\
&\pi_f(S)(U)=\pi'_f(\kappa^{-1}_{U, S^{-1}})|_{\frac{5}{2}-k}U-(v||(U-1))(T)=\pi'_f(1)|_{\frac{5}{2}-k}U+Q_f=Q_f.
\nonumber\end{align}
Since $\og_1$ is cuspidal, $\rho_{\og_1}(x)(z) \in \mathbb C_{k-\frac52}[z]$ for all $x \in H(2)\backslash \PSL_2(\mathbb Z)$ and hence, with \eqref{indep} and \eqref{rhofexp}, we deduce, by comparing the coefficients of $z^{-1}$ and $z^{k-\frac32}$ in \eqref{basic2}:
\begin{align}\label{indepext1}
&0=  \frac{a_0\left ( \alpha E_1+\beta E_2\right )}{k-\frac32}z^{k-\frac32}
+\frac{a_0 \left ( \alpha E_1|_{k-\frac12}S+\beta E_2|_{k-\frac12}S\right )}{k-\frac32}z^{-1} \\
&b_{k-\frac{3}{2}}z^{-1}=\frac{a_0 \left ( \alpha E_1|_{k-\frac12}T+\beta E_2|_{k-\frac12}T \right )}{k-\frac32}z^{k-\frac32}
+\frac{a_0 \left (\alpha E_1|_{k-\frac12}TS+\beta E_2|_{k-\frac12}TS\right )}{k-\frac32}z^{-1}.
\label{indepext2}\end{align}
Since $a_0(E_1)=1$ and $a_0(E_2)=1+2^{k-\frac12}$, \eqref{indepext1} implies $\alpha=-(1+2^{k-\frac12})\beta$. Further, the transformation law of $E(z)$ implies
$(E(\cdot /2)|_{k-\frac12}U)(z)=E((z-1)/2)$
and  $(E(2\cdot )|_{k-\frac12}U)(z)=2^{\frac12-k}E(z/2).$
Hence $a_0(E_2|_{k-\frac12}U)=2$ and, since $E_1|_{k-\frac12}U=E_1$,  \eqref{indepext2} gives
\begin{equation}\label{alpha}
\beta=\frac{(k-\frac32)b_{k-\frac32}}{1-2^{k-\frac12}} \quad \text{and $\alpha=\frac{(k-\frac32)(1+2^{k-\frac12})b_{k-\frac32}}{2^{k-\frac12}-1}$}
\end{equation}
To determine the ``cuspidal" part of $g_1(z)=g(z/2)$, we employ the generalised Haberland formula of \cite{PP}. To state it we review some notation: For $\sum a_n z^n, \sum b_n z^n \in \mathbb C^{k-\frac{5}{2}}[z]$, we set
 $$\langle \sum a_n z^n, \sum b_n z^n \rangle=\sum (-1)^{n}\binom{k-\frac{5}{2}}{n}^{-1} a_n b_{k-\frac{5}{2}-n}$$
and, then, for $f, g \in I_k$ with $f(x), g(x) \in \mathbb C_{k-\frac{5}{2}}[z]$ for all $x \in H(2) \backslash \PSL_2(\mathbb Z),$
$$\langle \langle f, g \rangle \rangle=
\sum_{x \in 
H(2) \backslash \PSL_2(\mathbb Z)} \langle f(x), g(x) \rangle$$
(we have renormalised this so that it is consistent with our normalisation of the Petersson scalar product.) 
Finally, to define the generalised Haberland pairing in $W \times W$, we note that, as shown in Section 8 of \cite{PP}, each $h \in W$ has a unique decomposition of the form 
$$h=\oh+\tilde h||(1-S) \quad \text{with $\tilde h(x):=c_x z^{k-\frac32}$ for some $c_x$ such that $c_{x}=c_{xT}$ ($x \in H(2) \backslash \PSL_2(\mathbb Z)$)}$$
and some $\oh \in I_k$ with $\oh(x) \in \mathbb C_{k-\frac{5}{2}}[z]$ for all $x \in H(2) \backslash \PSL_2(\mathbb Z).$ Then, for $h, k \in W$ we set
$$\{h, k\}:=\langle \langle \oh||(T-T^{-1}), \ok \rangle \rangle+
\langle \langle 2\tilde h||(T-T^{-1}), \ok \rangle \rangle+\langle \langle \oh, 2\tilde k||(T^{-1}-T)\rangle \rangle.$$
The generalisation of Haberland's formula proved in \cite{PP} (Theorem 8.6(c)), in the case we are considering here can be stated as
\begin{equation}\label{hab}-3(2i)^{k-\frac32}(f, g)=\{\rho^-_f, \bar\rho^+_g\} \quad \text{for $f, g, \in M_{k-\frac12}(H(2))$}
\end{equation}
where $(f, g)$ is a generalisation of the Petersson scalar product valid for not necessarily cuspidal modular forms. 
We will apply this formula to $g_1$ and to any $h \in S_{k-\frac12}(H(2)).$ In that case, we do not need the generalisation of the Petersson scalar product. The standard one suffices because $h$ is cuspidal. Then, \eqref{hab} simplifies as
\begin{equation}\label{habsim}-3(2i)^{k-\frac32}(\og_1, h)=-3(2i)^{k-\frac32}(g_1, h)=
\langle \langle \overset{o}{\rho_{g_1}}{}^{-}||(T-T^{-1}), \bar \rho_h^+ \rangle \rangle+
\langle \langle 2\tilde \rho_{g_1}^{-}||(T-T^{-1}), \bar \rho_h^+\rangle \rangle
\end{equation}
since $h$ is cuspidal and hence $\rho_h=\overset{o}{\rho_h}.$

We next determine $\overset{o}{\rho_{g_1}}{}^{-}$ and $\tilde \rho_{g_1}$: First, \eqref{basic}, combined with \eqref{indep}, \eqref{Lval} and \eqref{Q} imply that the coefficient $\alpha_n$ of $z^n$ in $\overset{o}{\rho_{g_1}}{}^{-}(x)$, for odd $n \in \{0, \dots, k-\frac52\}$, equals
$$ \alpha_n(x)=\begin{cases} i^{a+1}\left (\binom{k-2}{n} N^{\frac{n}{2}} \Lambda(f, a+1-n)+ \binom{k-2}{n+\frac12} N^{2k-a-\frac{3n}{2}-\frac{19}{4}}\Lambda(f, 2k-\frac72-a-n) \right ), \,  \text{if $x=1$} \\  
b_{k-\frac52-n}  \qquad \text{if $x=T$, and} \\b_{n} \qquad \text{if $x=U$.} 
\end{cases}$$
We will next verify that $$\tilde \rho_{g_1}(x)=\tilde \rho_{g_1}^-(x)=c_xz^{k-\frac32} \qquad \text{with $c_1=c_T=0$ and $c_U=b_{k-\frac32}.$}$$ Indeed, 
$\tilde \rho_{g_1}||(1-S)(x)=c_xz^{k-\frac32}+c_{xS}z^{-1}$ and hence it equals
$0$ for $x=1$, $b_{k-\frac{3}{2}}z^{-1}$, for $x=T$ and $b_{k-\frac{3}{2}}z^{k-\frac32}$, for $x=T$. These agree with the corresponding terms of $\rho_{g_1}^{-}(x)=\pi_f(S)^-(x)$ as given by \eqref{indep}, for all $x.$

We finally determine $\rho_h^+$ for $h \in S_{k-\frac12}(H(2))$: For each $x \in H(2)\backslash \PSL_2(\mathbb Z)$ and $n=0, \dots k-\frac{5}{2}$,  we denote by $r^{\pm}_{x, n}(h)$ the constants such that 
\begin{equation}\label{rhopol}
\rho^{+}_h(x)(z) = \sum_{n=0}^{k-\frac{5}{2}}\binom{k-\frac{5}{2}}{n} r^{+}_{x, n}(h) z^{k-\frac{5}{2}-n}.
\end{equation}
In particular, $r^{+}_{x, n}(h)=0$, if $n$ is odd. With this notation, for $x \in H(2)\backslash \PSL_2(\mathbb Z)$ and $0 \le n \le k-\frac{5}{2}$, we let $R^{+}_{x, n}$ be the unique element of $S_{k-\frac{1}{2}}(H(2))$ such that 
\[r^+_{x, n}(h)=(h, R^+_{x, n}) \quad \mbox{for all} \, \, h \in  S_{k-\frac{1}{2}}(H(2))\].

We are now ready to compute the terms in the RHS of \eqref{habsim}. We first note that
$H(2)xT=H(2)xT^{-1}$ for each $x \in H(2) \backslash \PSL_2(\mathbb Z)$ and hence
$$\overset{o}{\rho_{g_1}}{}^{-}||(T-T^{-1})(x)=\overset{o}{\rho_{g_1}}{}^{-}(xT)|_{\frac52-k}(T-T^{-1})=\sum_{\substack{\ell=0 \\ \ell \, \text{even}}}^{k-\frac52}s_{\ell}(xT)z^{\ell}$$
where 
\begin{equation}\label{s}
s_{\ell}(x)= 2\sum_{\substack{j=0 \\ j \, \text{odd}}}^{k-\frac52} \binom{j}{\ell} \alpha_j(x).
\end{equation}
Likewise, $$\tilde \rho_{g_1}^-||(T-T^{-1})(x)=c_{xT}\left ( (z+1)^{k-\frac32}-(z-1)^{k-\frac32} \right )=\sum_{\substack{\ell=0 \\ \ell \, \text{even}}}^{k-\frac52}2c_{xT} \binom{k-\frac32}{\ell}z^{\ell}.$$
(which is non-zero only for $x=U.)$

With the definition of the pairing $\langle \langle \cdot, \cdot \rangle \rangle$, equ. \eqref{habsim} then becomes
\begin{equation}\label{habsim2}-3(2i)^{k-\frac32}(\og_1, h)=
\sum_{x \in H(2)\backslash \PSL_2(\mathbb Z)} \sum_{\substack{\ell=0 \\ \ell \, \text{even}}}^{k-\frac52} \left (s_{\ell}(xT)
+2c_{xT}\binom{k-\frac32}{\ell} \right ) \overline{(h, R_{x, \ell}^+)}
\end{equation}
 By the non-degeneracy of the Petersson scalar product, we
deduce 
$$\og_1=\frac{-(2i)^{\frac32-k}}{3}\sum_{x \in H(2)\backslash \PSL_2(\mathbb Z)} \sum_{\substack{\ell=0 \\ \ell \, \text{even}}}^{k-\frac52} \left (s_{\ell}(xT)
+2c_{xT}\binom{k-\frac32}{\ell} \right )R_{x, \ell}^+.$$ 
Taking into account that $g(z)=g_1(2z)$ is of weight $k-\frac{1}{2}$ for $\Gamma_0^*(4)$, this equality implies  
\begin{theorem}\label{expl}
    Let $k \in \frac{1}{2}+\mathbb Z$ such that $k>\frac{5}{2}$ and $4|(k-\frac{5}{2})$. For each $f \in S_k(\Gamma_0(N))$ such that $f|_kW_N=f$ and for each $a \in [0, 2k-\frac92]$, the ``lift" $g \in M_{k-\frac{1}{2}}(\Gamma^{*}_0(4))$ of Prop. \ref{lift2} is given by
    $$g(z)=\frac{-(2i)^{\frac32-k}}{3}\sum_{x \in H(2)\backslash \PSL_2(\mathbb Z)} \sum_{\substack{\ell=0 \\ \ell \, \text{even}}}^{k-\frac52}  \left (s_{\ell}(xT)
+2c_{xT}\binom{k-\frac32}{\ell} \right )R_{x, \ell}^+(2z)+\alpha E_1(2z)+\beta E_2(2z)$$
    where $s_{\ell}(xT)$ and $\alpha, \beta$ are given by \eqref{s} and \eqref{alpha} respectively.
\end{theorem}


\begin{thebibliography}{99}
\bibitem{BCD}
R. Bruggeman, Yj. Choie, N. Diamantis \emph{Holomorphic Automorphic Forms
and Cohomology}, Memoirs of the AMS, 253; 1212 (2018)

\bibitem{Iw} H. Iwaniec, \emph{Topics in Classical Automorphic Forms} Graduate Studies in Mathematics, vol. 17, American Mathematical Society, Providence, RI, 1997. 

\bibitem{KR}
W. Kohnen, W. Raji, \emph{Special values of Hecke $L$-functions of
modular forms of half-integral weight and
cohomology}, Res Math Sci (2018) 5:22

\bibitem{N} D. Naidu, \emph{Categorical Morita Equivalence for Group-Theoretical Categories}, 
Comm. in Algebra (2007) 35:11, 3544-3565
 
\bibitem{NIST} F. Olver, D. Lozier, R. Boisvert, and
C. Clark, \textsl{NIST handbook of mathematical
functions} U.S.
Department of Commerce, National Institute of Standards
and Technology,
Washington, DC; Cambridge University Press, Cambridge, 2010.

\bibitem{PP}
 B. Pa\c{s}ol and A. Popa, {\it Modular forms and period polynomials}, Proc. London Math. Soc. (2013) {\bf 107} no. 4, 713--743.

\bibitem{Sh}
G.Shimura \emph{On modular forms of half-integral weight}, Ann. Math. 97, 440–481 (1973)


 
\end{thebibliography}
\end{document}